\documentclass{amsart}
\usepackage{amssymb, amsmath}
\newtheorem{theorem}{Theorem}
\newtheorem{definition}[theorem]{Definition}
\newtheorem{lemma}[theorem]{Lemma}
\newtheorem{corollary}[theorem]{Corollary}
\newtheorem{proposition}[theorem]{Proposition}
\newtheorem{remark}[theorem]{Remark}

\newcommand{\I}{\mbox{Id}}

\begin{document}
\title{Graded polynomial identities and exponential growth}
\author{Eli Aljadeff}
\address{Department of Mathematics, Technion-Israel Institute of
Technology, Haifa 32000, Israel} \email{aljadeff@tx.technion.ac.il}
\author{Antonio Giambruno}
\address{Dipartimento di Matematica ed Applicazioni,
Universit\`a di Palermo, Via Archirafi 34, 90123 Palermo, Italy}
\email{agiambr@unipa.it}
\author{Daniela La Mattina}
\address{Dipartimento di Matematica ed Applicazioni,
Universit\`a di Palermo, Via Archirafi 34, 90123 Palermo, Italy}
\email{daniela@math.unipa.it}

\keywords{graded algebra, polynomial identity, growth, codimensions}

\thanks{The first author was supported by the ISRAEL SCIENCE FOUNDATION
(grant No. 1283/08) and by the E. SCHAVER RESEARCH FUND. The second
and third authors were partially supported by MIUR of Italy}

\subjclass[2000]{16R10, 16W50, 16P90}

\begin{abstract} Let $A$ be a finite dimensional algebra over a field
of characteristic zero graded by a finite abelian group $G$. Here we
study a growth function related to the graded polynomial identities
satisfied by $A$ by computing the exponential rate of growth of the
sequence of graded codimensions of $A$. We prove that the
$G$-exponent of $A$ exists and is an integer related in an explicit
way to the dimension of a suitable semisimple subalgebra of $A$.
\end{abstract}

\maketitle

\section{Introduction}
Let $A$ be an associative algebra over a field $F$ of characteristic
zero graded by a finite abelian group $G.$ Graded polynomials and
graded polynomial identities are defined in a natural way and there
is a well established literature on the subject (see for instance
\cite{AK}, \cite{AHN}, \cite{BD}, \cite{BZ}, \cite{BJ},
\cite{GZbook}). Here we want to study a growth function related to
the graded identities of $A$ and compare it with the ordinary case.
More precisely let $F\langle X,G \rangle$ be the free $G$-graded
algebra on a countable set and $\I^G(A)$ the ideal of $G$-graded
identities of $A.$ Then one considers the relatively free $G$-graded
algebra $F\langle X,G \rangle/\I^G(A)$ and denotes by $c^G_n(A)$ the
dimension of the subspace of multilinear elements in $n$ free
generators. The sequence $c^G_n(A),$ $n=1, 2,\ldots,$ is called the
sequence of $G$-codimensions of $A$ and is bounded from above by
$|G|^nn!$ Nevertheless when $A$ satisfies an ordinary (non graded)
polynomial identity (in this case we say that $A$ is a PI-algebra),
one can compare $c^G_n(A)$ and $c_n(A),$ $n=1, 2, \ldots,$ the
ordinary codimension sequence of $A$ and it turns out that
$c^G_n(A)\leq |G|^n c_n(A)$ holds for all $n$ (\cite{gia-reg}).

Since in the presence of an ordinary identity the sequence of
codimensions is exponentially bounded (\cite{reg1}), it turns out
that also $c^G_n(A)$ is exponentially bounded and our aim in this
paper is to determine the exponential rate of growth of this
sequence. Now, for the ordinary codimensions, in case $A$ is a
PI-algebra, in \cite{gz1} and \cite{gz2} it was proved that
$\lim_{n\to \infty}\sqrt[n]{c_n(A)}=\exp(A)$ exists and is an
integer called the PI-exponent of $A.$ Here we shall prove that if
$A$ is a finite dimensional $G$-graded algebra,
$$
\exp^G(A)=\lim_{n\to
\infty}\sqrt[n]{c^G_n(A)}
$$
exists and is an integer. Moreover $\exp^G(A)$ can be explicitly
computed and it turns out to be equal to the dimension of a suitable
semisimple $G$-graded subalgebra of $A$ over an algebraically closed
field. In case $G$ is of order two, this result was proved in
\cite{benanti-giamb-pipitone}.

For non associative algebras satisfying a polynomial identity, the
corresponding codimensions in general are not exponentially bounded.
For Lie algebras an interesting scale of functions between
exponential and factorial have been provided by Petrogradsky in
\cite{P}. Even when the codimensions are exponentially bounded, the
exponential rate of growth can be non integer. In \cite{GMZ2} the
authors constructed for any real number $\alpha>1$ an algebra whose
exponential growth of the codimensions is equal to $\alpha$. For Lie
algebras in \cite{Z} it was shown that the PI-exponent of a finite
dimensional Lie algebra exists and is an integer, but in \cite{ZM}
an example was given of a Lie algebra whose PI-exponent is not an
integer.

The proof of the main result of the paper heavily relies on the
representation theory of the symmetric group and on the computation
of the asymptotics for the degrees of the irreducible
$S_n$-representations, $n\ge 1$, associated to Young diagrams of
bounded height. It is worth to point out that along the proof we
construct a family of multialternating $G$-graded polynomials for a
$G$-simple algebra which is of independent interest.

\section{Preliminaries}

Throughout the paper $F$ will denote a field of characteristic zero,
$G$ a finite abelian group and $A$  a $G$-graded algebra over $F$.
If $G=\{g_1=1, \ldots, g_s\}$, we write $A=\bigoplus_{i=1}^s
A_{g_i}$, where $A_{g_i}A_{g_j}\subseteq A_{g_ig_j}$, for all $i,
j=1, \ldots, s$, and the $A_{g_i}$'s are the homogeneous components
of $A$.

Let  $F\langle X, G\rangle$ be the free associative $G$-graded
algebra on a countable set $X$ over $F$. Write $X$ as
$$
X=\bigcup_{i=1}^s X_{g_i},
$$
where  the sets $X_{g_i}=\{x_{1,g_i}, x_{2,g_i}, \ldots \}$ are
disjoint, and the elements of $X_{g_i}$ have homogeneous degree
$g_i.$ In general, given a  monomial $x_{i_1,g_{j_1}}\cdots
x_{i_t,g_{j_t}}$ its homogeneous degree is $g_{j_1}\cdots g_{j_t}.$
If we denote by ${\mathcal F}_{g_i}$ the subspace of $F\langle X, G
\rangle$ spanned by all monomials in the variables of $X$ having
homogeneous degree $g_i,$ then $F\langle X,G \rangle=\bigoplus_i
{\mathcal F}_{g_i}$ is the natural $G$-grading of $F\langle X,G
\rangle.$ In fact any graded map $X\to B$, where $B$ is any
$G$-graded algebra, can be extended to a homomorphism of graded
algebras $F\langle X, G\rangle \to B$.

Recall that an element $f=f(x_{1,g_1}, \ldots, x_{t_1,g_1}, \ldots,
x_{1,g_s}, \ldots, x_{t_s,g_s})$ of $F\langle X, G \rangle$ is
called a graded polynomial, and we write $f\equiv 0$ in case $f$
vanishes under all graded substitutions in $A$. Thus $f\equiv 0$ if
$f(a_{1,g_1}, \ldots, a_{t_1,g_1}, \ldots, a_{1,g_s}, \ldots,
a_{t_s,g_s})=0$, for all $a_{1,g_i}, \ldots, a_{t_i,g_i} \in
A_{g_i},$ $i=1,\ldots, s$.

We denote by $\I^G(A)=\{f\in F\langle X, G \rangle \mid \ f\equiv 0
\ \mbox{on}\ A\}$ the ideal of graded identities of $A.$ It is
easily checked that $\I^G(A)$ is invariant under all graded
endomorphisms of $F\langle X, G \rangle$. The object of our study
will be $F\langle X, G \rangle/\I^G(A)$, the free algebra of the
variety of $G$-graded algebras generated by $A$. It is well known
that in characteristic zero such algebra is determined by its
multilinear components. Hence for $n\geq 1$ we denote by
$$
P_n^G=\mbox{span}_F\{x_{\sigma(1),g_{i_1}}\cdots x_{\sigma(n),
g_{i_n}}\ | \ \sigma \in S_n, g_{i_1}, \ldots, g_{i_n}\in G \}
$$
the space of multilinear $G$-graded  polynomials in the variables
$x_{1,g_{i_1}}, \ldots, x_{n,g_{i_n}}$, where $g_{i_j}\in G$ and
accordingly we set
$$
P_n^G(A)= \frac{P^G_n}{P^G_n\cap \I^G(A)}.
$$
The non-negative integer
$$
c^G_n(A)=\dim_FP_n^G(A), \ n\geq 1,
$$
is called the $n$-th $G$-graded codimension of $A$ and our aim here
is to study the asymptotic behavior of the sequence $c^G_n(A), \
n=1, 2, \ldots$ and to compare it with the ordinary case.

Notice that if for $i\ge 1$ we set $x_i=x_{i,g_1}+\cdots
+x_{i,g_s}$, then the free algebra on the countable set
$Y=\{x_1,x_2, \ldots\}$ is naturally embedded in $F\langle X,
G\rangle$ and we can regard the ordinary identities of $A$ as a
special kind of graded identities.

% in case we consider $A$ with the trivial $G$-grading,
%then $F\langle X, G\rangle =F\langle X\rangle$ is the free
%associative algebra on $X$ and graded polynomials or graded
%identities of $A$ become ordinary polynomials and ordinary
%identities of $A$. When $A$ is endowed with a fixed $G$-grading, it
%is useful to regard an ordinary polynomial identity in $x_1, \ldots,
%x_n$ as a graded polynomial identity in the variables
%$x_i=x_{i,g_1}+\cdots +x_{i,g_s}, i=1, \ldots, n.$

We denote by $P_n$ the space of multilinear polynomials in the
variables $x_1, \ldots, x_n$ and by $\I(A)=\{f\in F\langle Y\rangle
\mid f\equiv 0 \ \mbox{in}\ A\}$ the T-ideal of ordinary polynomial
identities of $A$. Then $c_n(A)=\dim_F{P_n}/({P_n\cap \I(A)})$ is
the $n$-th ordinary codimension of $A$.

It is well-known that for a general algebra $A$ satisfying an
ordinary polynomial identity (non necessarily finite dimensional),
the sequence of codimensions is exponentially bounded (\cite{reg1}).

The same conclusion about the $G$-codimensions cannot be drawn when
$A$ satisfies a $G$-graded polynomial identity, but not an ordinary
identity. For instance, if $|Y|>1,$ the free algebra $A=F\langle
Y\rangle$ with trivial $G$-grading satisfies the graded identities
$x_g\equiv 0$ for $g\neq 1,$  does not satisfy any ordinary identity
and $c^G_n(A)=n!.$ Nevertheless in case $A$ satisfies an ordinary
polynomial identity (e.g., is finite dimensional), the two sequences
are related by the following inequalities (\cite{gia-reg})
\begin{equation}\label{eq1}
c_n(A) \leq c_n^G(A) \leq |G|^n c_n(A), \ n\geq 1.
\end{equation}
Hence the sequence of $G$-codimensions is exponentially bounded and
we can try to compute its exponential rate of growth. In \cite{gz1}
and \cite{gz2} it was proved that if $A$ is any algebra satisfying
an ordinary polynomial identity,  then $\lim_{n\to
\infty}\sqrt[n]{c_n(A)}=\mbox{exp}(A)$ exists and is an integer
called the PI-exponent of the algebra $A$.

At the light of (\ref{eq1}), for a $G$-graded algebra satisfying an
ordinary polynomial identity, one can ask if $\lim_{n\to
\infty}\sqrt[n]{c_n^G(A)}$ exists and is an integer. Here we shall
give a positive answer to the above question when $A$ is a finite
dimensional algebra (hence it satisfies an ordinary polynomial
identity).

For fixed $n=n_1+\cdots+n_s>0$ let $P_{n_1, \ldots, n_s}\subseteq
P^G_n$ denote the subspace of multilinear graded polynomials in
which  the first $n_1$ variables are homogeneous of degree $g_1,$
next $n_2$ variables are homogeneous of degree $n_2$ and so on. Also
let
$$
P_{n_1, \ldots, n_s}(A)=\frac{P_{n_1, \ldots, n_s}}{P_{n_1,
\ldots, n_s}\cap \I^G(A)}
$$
and
$$
c_{n_1,\ldots, n_k}(A)=\dim P_{n_1, \ldots, n_s}(A).
$$
By counting dimensions it follows that
\begin{equation}\label{cG}
c_n^G(A)=\sum_{n_1+\cdots+n_s=n}\binom{n}{n_1, \ldots,
n_s}c_{n_1,\ldots, n_s}(A),
\end{equation}
where $\displaystyle\binom{n}{n_1, \ldots, n_k}=\frac{n!}{n_1!\cdots
n_s!}$\  denotes the generalized binomial coefficient. We shall use
(\ref{cG}) in order to deduce an upper and lower bound for
$c_n^G(A)$ from an upper and lower bound for $c_{n_1,\ldots,
n_k}(A)$.

Our methods will be based on the representation theory of the
symmetric group $S_n$. Now, the space $P_{n_1, \ldots, n_s}$ is
naturally endowed with a structure of $S_{n_1}\times \cdots \times
S_{n_s}$-module in the following way: $S_{n_1}\times \cdots \times
S_{n_s}$ acts on the left on $P_{n_1, \ldots, n_s}$ by permuting the
variables of homogeneous degree $g_1, \ldots, g_s,$ separately
($S_{n_i}$ permutes the variables of degree $g_i$). Since $\I^G(A)$
is invariant under this action,  $P_{n_1, \ldots, n_s}(A)$ inherits
a structure of $S_{n_1}\times \cdots \times S_{n_s}$-module and we
denote by $\chi_{n_1,\ldots, n_k}(A)$ its character.

By complete reducibility we can write
\begin{equation}\label{multipl}
\chi_{n_1,\ldots, n_s}(A)=\sum_{\langle \lambda\rangle\vdash
n}m_{\langle \lambda\rangle}\chi_{\lambda(1)}\otimes\cdots \otimes
\chi_{\lambda(s)},
\end{equation}
where $\langle \lambda \rangle=(\lambda(1), \ldots,
\lambda(s))\vdash (n_1, \ldots, n_s)$ is a multipartition of $n,$
i.e, $\lambda(1)\vdash n_1 ,\ldots, \lambda(s)\vdash n_s,$
$n_1+\cdots+n_s=n,$ and  $m_{\langle \lambda\rangle}\ge 0$ is the
multiplicity of $\chi_{\lambda(1)}\otimes\cdots \otimes
\chi_{\lambda(s)}$ in $\chi_{n_1,\ldots, n_k}(A)$. We call
$\chi_{n_1,\ldots, n_k}(A)$ the $(n_1, \ldots, n_s)$-cocharacter of
$A.$ A basic fact that we shall need in what follows is that the
multiplicities $m_{\langle \lambda\rangle}$ in (\ref{multipl}) are
polynomially bounded.

\begin{remark} \label{rem1}
There exist constants $a,b$ such that for all $n\ge 1$, $m_{\langle
\lambda\rangle}\le an^b$ in (\ref{multipl}).
\end{remark}

\begin{proof}
Since $G$ is a finite abelian group, there is a well-known duality
between $G$-gradings and  $G$-actions by automorphisms on the
associative algebra $A$. Let $\hat G$ be the group of linear
characters of $G$. Then $\hat G \cong G$ and one defines an action
of $\hat G$ on $A$ as follows: if $a= \sum_{g \in G}  a_g$ with $a_g
\in A_g$ and $\psi\in \hat G$, then $\psi (a)= \sum_{g \in G} \psi
(g) a_g$.

If we apply the above duality to the free $G$-graded algebra
$F\langle X, G\rangle$, we may regard $F\langle X, G\rangle$ as the
free algebra with $\hat G$-action $F\langle Y|\hat G\rangle$, on a
countable set $Y$.

 Recall that $F\langle Y | \hat G \rangle$ is freely
generated by the set $\{ y^\psi=\psi(y) | y \in Y, \psi \in \hat
G\}$ on which $\hat G$ acts as follows: $f(y_1^{\psi_1}, \dots
,y_n^{\psi_n})^\psi=f(y_1^{\psi\psi_1}, \dots ,y_n^{\psi\psi_n})$.
The elements of $F\langle Y | \hat G \rangle$  are called $\hat
G$-polynomials and the ones vanishing in $A$ are the $\hat
G$-identities of $A$. Let $\I(A|\hat G)$ be the ideal of $\hat
G$-identities of $A$. By \cite[Proposition 2]{GMZ} $\I(A|\hat
G)=\I^G(A)$. Moreover
$$
P_n^G= \mbox{span}\{y_{\sigma(1)}^{\psi_1}\cdots
y_{\sigma(n)}^{\psi_n} \mid \sigma\in S_n, \psi_i\in \hat G\}.
$$
Let $H=\hat G \wr S_n$ be the wreath product of $\hat G$ and $S_n$.
The group $H$ acts on the left on $P_n^G$ by setting for
$f(y_1,\ldots, y_n)\in P_n^G$ and $(\psi_1, \ldots, \psi_n,\sigma)
\in H,$ then
$$
(\psi_1, \ldots, \psi_n,\sigma) f(y_1,\ldots, y_n)=
f(y_{\sigma(1)}^{\psi_{\sigma(1)}^{-1}},\dots
,y_{\sigma(n)}^{\psi_{\sigma(n)}^{-1}}).
$$
Hence the ideals of $\hat G$-identities of $F\langle Y | \hat
G\rangle$ are invariant under the left $H$-action.
 This makes $P_n^G(A)=P_n^G/(P_n^G\cap \I(A|\hat G))$ a left
$H$-module.

Now the irreducible $H$-characters are indexed by multipartitions of
$n$. For $\langle \lambda \rangle$ a multipartition of $n$, let
$\chi_{\langle \lambda \rangle}$ be the corresponding irreducible
$H$-character. The $H$-character of $P_n^G(A)$ is
\begin{equation} \label{mult2}
\chi_n(A|G)= \sum_{\langle \lambda \rangle\vdash n} m'_{\langle
\lambda \rangle}\chi_{\langle \lambda \rangle}.
\end{equation}
The connection between the $H$-character of $P_n^G(A)$ given in
(\ref{mult2}) and the $S_{n_1}\times \cdots \times
S_{n_s}$-character $\chi_{n_1,\ldots, n_s}(A)$ given in
(\ref{multipl}), is given in \cite[Theorem 4]{GMZ}: for all
multipartitions $\langle \lambda \rangle=(\lambda(1), \dots
,\lambda(s))$ where $\lambda (1)\vdash m_1, \dots ,\lambda (s)\vdash
m_s$, we have $ m_{\langle \lambda \rangle}= m'_{\langle \lambda
\rangle}$. But then, since by \cite{berele}  the multiplicities
$m'_{\langle \lambda \rangle}$ in (\ref{mult2}) are polynomially
bounded, the proof follows.
\end{proof}

%Finally we register another basic fact that we shall use in what
%follows. Recall that if $\lambda\vdash n$, then $\lambda'$ is the
%conjugate partition of $\lambda$. We shall also make the
%identification  $F(S_{n_1}\times \cdots \times S_{n_s})\equiv
%FS_{n_1}\otimes \cdots \otimes FS_{n_s}$.

In what follows if $\lambda\vdash n$ is a partition of $n$ we shall
denote by  $\lambda' = (\lambda'_1, \ldots, \lambda'_r) \vdash n$
the conjugate partition of $\lambda$.

\begin{remark} \label{rem2}
Let $\langle \lambda \rangle=(\lambda(1), \dots ,\lambda(s))$ be a
multipartition of $n$ and $N_{\langle\lambda\rangle} \not \subseteq
\I^G(A)$ an irreducible $S_{n_1}\times \cdots \times
S_{n_s}$-module. Then $N_{\langle\lambda\rangle}= F(S_{n_1}\times
\cdots \times S_{n_s})f$ for some polynomial $f\in P_{n_1, \ldots,
n_s}$ such that for every $j\ge 1,$ $f$ is $(\lambda(1)'_j, \ldots,
\lambda(s)'_j)$-alternating, i.e., $f$ is alternating into $s$
disjoint sets of variables $X^{j}_{g_1}, \ldots, X^{j}_{g_s}$ of
homogeneous degree $g_1, \ldots, g_s,$ respectively.
\end{remark}

\begin{proof}
We consider the decomposition of the group algebra $F(S_{n_1}\times
\cdots \times S_{n_s})= \oplus_{\langle\lambda\rangle\vdash
n}I_{\langle\lambda\rangle}$ where $I_{\langle\lambda\rangle}$ is
the two-sided ideal associated to the multipartition
$\langle\lambda\rangle$. Moreover every $I_{\langle\lambda\rangle}$
decomposes as
$$
\sum_{T_{\lambda(1)},\ldots,T_{\lambda(s)}} F(S_{n_1}\times \cdots
\times S_{n_s})e_{T_{\lambda(1)}}\cdots e_{T_{\lambda(s)}}
$$
where, for every $i$, $e_{T_{\lambda(i)}}$ is the essential
idempotent of $FS_{n_i}$ corresponding to the tableau
$T_{\lambda(i)}$ of shape $\lambda(i)$. Recall that
$e_{T_{\lambda(i)}}=\left(\sum_{\sigma\in
R_{T_{\lambda(i)}}}\sigma\right)\left(\sum_{\tau\in
C_{T_{\lambda(i)}}}(\mbox{sgn}\, \tau)\tau\right),$ where
$R_{T_{\lambda(i)}}$ and $C_{T_{\lambda(i)}}$ are the subgroups of
$S_{n_i}$ stabilizing the rows and the columns of $T_{\lambda(i)}$,
respectively.

Since $I_{\langle\lambda\rangle}N_{\langle\lambda\rangle} =
N_{\langle\lambda\rangle}\not \subseteq \I^G(A)$, there exist
tableaux $T_{\lambda(1)},\ldots,T_{\lambda(s)}$ and $f\in
N_{\langle\lambda\rangle}$ such that $F(S_{n_1}\times \cdots \times
S_{n_s}) e_{T_{\lambda(1)}}\cdots e_{T_{\lambda(s)}}f =
N_{\langle\lambda\rangle}$. If we set
$$
r_i= \sum_{\tau\in C_{T_{\lambda(i)}}}(\mbox{sgn}\, \tau)\tau,
$$
$1\le i\le s$ then $r_1e_{T_{\lambda(1)}} \cdots
r_se_{T_{\lambda(s)}}f$ is the required polynomial.

\end{proof}

\section{The general setting}

Throughout this and the next sections $A$ will be a finite
dimensional $G$-graded algebra over a field $F$ of characteristic
zero, where $G$ is a finite abelian group. By the Wedderburn-Malcev
theorem (\cite{GZbook}), we can write $$A=B+J$$ where $B$ is a
maximal semisimple  subalgebra of $A$ and $J=J(A)$ is its Jacobson
radical. It is well-known that $J$ is a graded ideal, moreover by
\cite{} we assume, as we may, that $B$ is a $G$-graded subalgebra of
$A$. Hence we can write $$B=B_1\oplus\cdots\oplus B_m$$ where $B_1,
\ldots, B_m$ are $G$-graded simple algebras. We remark that the
$G$-codimensions of $A$ do not change upon extension of the base
field (see \cite{GZbook}). Hence, in order to compute the
exponential rate of growth of the $G$-codimensions of $A$ we assume
that the field $F$ is algebraically closed. We start by making a
definition

\begin{definition}
The $G$-dimension of $A$ is  $G$-$\dim A=(p_1, \ldots, p_s),$ where
$p_i=\dim A_{g_i}, \ 1\le i\le s=|G|$.
\end{definition}

We say that a semisimple algebra $C=C_1\oplus\cdots\oplus C_k$,
where $C_1, \ldots, C_k\in \{B_1, \ldots, B_m\}$ are distinct, is
admissible if $C_1JC_2J\cdots C_{k-1}JC_k \ne 0$. We then define
\begin{equation} \label{d}
d(A)= \mbox{max}\ (\dim C)
\end{equation}
 where $C$ runs over all
admissible subalgebras of $B$.

In Theorem \ref{mainth} below we shall prove that $d(A)$ coincides
with $\lim_{n\to\infty}\sqrt[n]{c_n^G(A)}$.

Next we introduce multialternating polynomials that will be
essential tools throughout the paper.

Let $t>0$ be an integer and, for $i=1, \ldots s$, define
$$X_{g_i}^j=\{x_{1,g_i}^j,\ldots,
x_{m_i,g_i}^j\}\subseteq X_{g_i}, \ j=1, \ldots, t,$$  $t$ distinct
sets of  $m_i\geq 0$ variables of homogeneous degree $g_i$. Let also
$Y\subseteq \bigcup_{i=1}^s X_{g_i}$  be another set of homogeneous
variables disjoint from the previous sets.

Also, let $f=f(X^1_{g_1}, \ldots, X^1_{g_s}, \ldots, X^t_{g_1},
\ldots, X^t_{g_s},Y)\in F\langle X,G \rangle$ be a multilinear
graded polynomial in the variables from the sets $X^j_{g_i}$ and
$Y,$ $i=1, \ldots, s$ and $j=1, \ldots, t.$

\begin{definition} If $f$ is alternating in the indeterminates of
each set $X^j_{g_i},$ then we say $f$  $t$-fold $(m_1, \ldots,
m_s)$-alternating. In case $t=1$  we simply say that $f$  is  $(m_1,
\ldots, m_s)$-alternating.
\end{definition}

Throughout we make also the convention that $\bar x$ is an
evaluation of the variable $x$ in $A$. Accordingly, $\bar X^i_{g_j}$
stands for an evaluation of the set $X^i_{g_j}$ into $A_{g_j}$.

\begin{lemma}
Let $d=d(A)$ be the integer defined above and let $$f=f(X^1_{g_1},
\ldots, X^1_{g_s}, \ldots, X^t_{g_1}, \ldots, X^t_{g_s},Y)$$ be a
polynomial $t$-fold $(m_1, \ldots, m_s)$-alternating. If
$m_1+\cdots+m_s>d$ then
$$
f(\bar X^1_{g_1}, \ldots, \bar X^1_{g_s}, \ldots, \bar X^t_{g_1}, \ldots,
\bar X^t_{g_s},\bar Y)=0
$$
for all evaluations $\bar X^j_{g_i}$ in $B_{g_i}$ and $\bar Y$ in
$A.$
\end{lemma}

\begin{proof}
Being $f$  multilinear it is clearly enough to evaluate $f$ into
elements of a basis of $A$. Let $\mathcal{A}=\mathcal{B}\cup
\mathcal{J}$ be a homogeneous basis of $A$ where
$\mathcal{B}=\cup_{i=1}^s \mathcal{B}_{g_i}$ is a basis of $B$ and
$\mathcal{J}=\cup_{i=1}^s \mathcal{J}_{g_i}$ is a basis of $J.$
Hence for $i=1, \ldots, s,$ $\mathcal{B}_{g_i}$ and
$\mathcal{J}_{g_i}$ are bases of $B_{g_i}$ and $J_{g_i},$
respectively.

Suppose $f$ does not vanish in $A$. Then, if we evaluate $f$ on $B,$
since $B_iB_j=0$ for $i\neq j,$ in order to get a non-zero value we
have to evaluate $f$ only on a $G$-simple component, say $B_r.$
Since $\dim B_r\leq d$ and $m_1+\cdots+m_s>d$ there exists $1\leq
j\leq m$ such that $m_j>\dim (B_r)_{g_j}.$ Hence $f$ vanishes on a
basis of $B_r,$ being alternating on $m_j$ elements of homogeneous
degree $g_j.$

Therefore, in order to get a non-zero value of $f$ we should
evaluate at least one variable in $J.$ Moreover we may assume that
for every $i=1, \ldots, s,$ $\dim B_{g_i}\geq m_i.$  But then the
non-zero monomials of $f$ evaluate into $B_{i_1}JB_{i_2}J\cdots
JB_{i_l}$ or $JB_{i_1}JB_{i_2}J\cdots JB_{i_l},$ or
$B_{i_1}JB_{i_2}J\cdots JB_{i_l}J$ or $JB_{i_1}JB_{i_2}J\cdots
JB_{i_l}J,$ where $B_{i_1}, \ldots, B_{i_l}\in \{B_1, \ldots, B_m\}$
are not necessarily distinct $G$-simple components. Since  $\dim
(B_{i_1}+\cdots +B_{i_l})\geq m_1+\cdots +m_s> d$ we get that
$B_{i_1}JB_{i_2}J\cdots JB_{i_l}=0$ and we reach a contradiction.
\end{proof}

\begin{lemma}
Let $\langle \lambda \rangle=(\lambda(1), \ldots, \lambda(s))\vdash
(n_1, \ldots, n_s)$ be a multipartition of $n$ and let $N_{\langle
\lambda \rangle}$ be an irreducible $S_{n_1}\times\cdots \times
S_{n_s}$-module participating with non-zero multiplicity in $P_{n_1,
\ldots, n_s}(A)$. Then
\begin{itemize}
\item[1)]
$\lambda(i)_1'\le \dim A_{g_i}, \ 1\le i\le s$.
\item[2)]
$\lambda(1)'_{l+1}+\cdots+\lambda(s)'_{l+1}\le d=d(A),$ where
$J^{l+1}=0,$
\item[3)]
$\dim N_{\langle \lambda \rangle} \le n^a
(\lambda(1)'_{l+1})^{n_1}\cdots (\lambda(s)'_{l+1})^{n_s}$, for some
$a\ge 1.$
\end{itemize}
\end{lemma}

\begin{proof}
By Remark \ref{rem2}, $N_{\langle\lambda\rangle}=
F(S_{n_1}\times\cdots \times S_{n_s})f$ where $f\in P_{n_1, \ldots,
n_s}$ is not an identity of $A$ and  for $j\ge 1,$ $f$ is
$(\lambda(1)'_j, \ldots, \lambda(s)'_j)$-alternating, i.e., $f$ is
alternating into $s$ disjoint sets of variables $X^{j}_{g_1},
\ldots, X^{j}_{g_s}$ of homogeneous degree $g_1, \ldots, g_s,$
respectively. Since $f$ is not an identity for $A,$ we must have
$\lambda(k)'_j\le \dim A_{g_k},$ for all $k=1, \ldots, s$ and $j\ge
1.$ This proves the first part of the lemma.

Suppose by contradiction that
 $\lambda(1)'_{l+1}+\cdots+\lambda(s)'_{l+1} > d.$
In particular, this says that
$\lambda(1)'_{j}+\cdots+\lambda(s)'_{j} > d$ for $j=1, \ldots, l+1$.
Since by hypothesis $f$  is not an identity of $A$, and $f$ is
$(\lambda(1)'_j, \ldots, \lambda(s)'_j)$-alternating, by the
previous lemma, in order to get a non-zero value, for $j=1, \ldots,
l+1,$ at least one variable in the sets $$X^{j}_{g_1}, \ldots,
X^{j}_{g_s}$$ must be evaluated in $J$. But $J^{l+1}=0$ says that
$f$ vanishes on $A$, a contradiction.

Let $\dim A_{g_1}=p_1, \ldots,  \dim A_{g_s}= p_s.$  From the hook
formula giving the degree of an irreducible $S_n$-character (see
\cite{JK}), we get that for $j=1, \ldots, s$,
$$\deg \chi_{\lambda(j)}\leq n^{p_jl}{(\lambda(j)'_{l+1})}^{n_j}.$$
Hence
$$
\dim N_{\langle \lambda\rangle} =\deg \chi_{\lambda(1)}\cdots \deg
\chi_{\lambda(s)}\le n^a {(\lambda(1)'_{l+1})}^{n_1}\cdots
{(\lambda(s)'_{l+1})}^{n_s},$$ and we are done.
\end{proof}

We are now in a position to determine a useful  upper bound for
$c_n^G(A)$.

\begin{lemma} \label{lem6}
There exist constants $a,t$ such that $c_n^{G}(A) \le an^td^n$.
\end{lemma}

\begin{proof} By part 2) of the previous lemma we can write
$$
\chi_{n_1,\ldots, n_k}(A)=\sum_{\langle \lambda\rangle\vdash n\atop
0\le \lambda(1)'_{l+1}+\cdots+\lambda(s)'_{l+1} \le d}m_{\langle
\lambda\rangle}(\chi_{\lambda(1)}\otimes\cdots \otimes
\chi_{\lambda(s)}).
$$
Hence by part 3) of the previous lemma,
$$
c_{n_1,\ldots, n_k}(A)\leq\sum_{\langle \lambda\rangle\vdash n\atop
0\le \lambda(1)'_{l+1}+\cdots+\lambda(s)'_{l+1} \le d}m_{\langle
\lambda\rangle}n^\alpha {(\lambda(1)'_{l+1})}^{n_1}\cdots
{(\lambda(s)'_{l+1})}^{n_s},
$$
for some constant $\alpha$. By Remark \ref{rem1} the multiplicities
$m_{\langle\lambda\rangle}$ are polynomially bounded; hence by
(\ref{cG}) we obtain
$$
c_n^G(A)=\sum_{n_1+\cdots+n_k=n}\binom{n}{n_1, \ldots,
n_k}c_{n_1,\ldots, n_k}(A)$$ $$\le an^b \sum_{0\leq
t_1+\cdots+t_s\le d} \sum_{n_1+\cdots+n_k=n}\binom{n}{n_1, \ldots,
n_k} t_1^{n_1}\cdots t_s^{n_s}
$$
$$
=  an^b \sum_{0\leq t_1+\cdots+t_s\le d} (t_1+\cdots+t_s)^n \le
an^bd^{n+s}.
$$
\end{proof}

\section{$G$-simple algebras and multialternating polynomials}

Throughout this section $A=\oplus_{g\in G}A_{g}$ will be a
$G$-simple algebra over an algebraically closed field $F$ where $G$
is a finite abelian group. Our aim is to construct a $t$-fold $(p_1,
\ldots, p_s)$-alternating polynomial for every even integer $t>0$,
where $(p_1, \ldots, p_s)=G$-$\dim A$.

The construction relies on a structure theorem proved by Bahturin,
Seghal and Zaicev in which they fully describe the $G$-grading on a
$G$-simple algebra in terms of fine and elementary gradings. It
should be mentioned that their result holds for arbitrary groups
(and not only for finite abelian groups).

\begin{theorem}[\cite{BSZ}]\label{BSZ} Let $A$ be a $G$-simple algebra. Then there exists
a subgroup $H$ of $G$, a $2$-cocycle $f:H\times H\longrightarrow
F^{*}$ where the action of $H$ on $F$ is trivial, an integer $k$ and
a $k$-tuple $(g_1=1,g_2,\ldots,g_k)\in G^k$ such that $A$ is
$G$-graded isomorphic to $C=F^{f}H\otimes M_{k}(F)$ where $C_g=
span_F\{b_h \otimes E_{i,j}: g=g_i^{-1}hg_j\}$. Here $b_h \in
F^{f}H$ is a representative of $h\in H$ and $E_{i,j}\in M_{k}(F)$ is
the $(i,j)$ elementary matrix.
\end{theorem}

We start with an analysis of a twisted group algebra $F^{f}H$ where
$H$ is a finite abelian group. This is well known but we include it
here for the reader convenience.

Let $H$ be a finite abelian group and $f:H\times H\longrightarrow
F^{*}$ a $2$-cocycle. Consider the function
$\alpha=\alpha_{f}:H\times H\longrightarrow F^{*}$ defined by
$\alpha(h_{1},h_{2})=f(h_{1},h_{2})f^{-1}(h_{2},h_{1})$. It follows
that if $\{u_h\}_{h\in H}$ is a set of representatives in $F^{f}H$
of the elements of $H$ then
$u_{h_1}u_{h_2}=\alpha(h_1,h_2)u_{h_2}u_{h_1}$ (roughly speaking the
function $\alpha$ is determined by the commuting relations of
representatives of elements of $H$ in $F^{f}H$). The following facts
follow directly from the definitions:
\begin{enumerate}
\item
the function $\alpha$ is determined by the commuting relations of
representatives in $F^{f}H$ of generators of $H$.
\item
if $f^{'}$ is a $2$-cocycle cohomologous to $f$ (that is there
exists a function $\lambda:H\longrightarrow F^{*}$ such that for
every $h_1, h_2$ in $H$,
$f^{'}(h_1,h_2)=\lambda_{h_1}\lambda_{h_2}\lambda_{h_{1}h_{2}}^{-1}f(h_1,h_2)$)
then $\alpha_{f}=\alpha_{f^{'}}$.
\end{enumerate}
It follows that the correspondence $f\longrightarrow \alpha_{f}$
induces a homomorphism from $H^{2}(H,F^{*})$ to
$Hom(\bigwedge^{2}H,F^{*})$. Here, $\bigwedge^{2}H$ denotes the
Schur multiplier of the group $H$ and is isomorphic to $(H
\otimes_{\mathbb{Z}} H)/W$ where $W$ is the subgroup generated by
all $\{h\otimes h \mid h \in H\}$.

\begin{proposition}
The map $\alpha$ is an isomorphism.
\end{proposition}
\begin{proof}
 See \cite[Chapter V]{br}.
\end{proof}

It follows that if $H\cong C_{n_1}\times\cdots \times C_{n_q}=
\langle h_1\rangle \times\cdots\times\langle h_q\rangle$ is a
decomposition of $H$ into cyclic groups of order $n_1,\ldots,n_q,$
respectively,  a cohomology class in $H^{2}(H,F^{*})$ is determined
by the value of $\alpha$ on pairs of generators $(h_i,h_j)$ for
$1\leq i<j\leq q$. In order to determine a cohomology class one may
choose for each $1\leq i<j\leq q$ an arbitrary root of unity whose
order divides the $g.c.d(n_i,n_j)$. Furthermore, different choices
determine different cohomology classes. We have therefore that the
cohomology classes are in $1-1$ correspondence with $\binom q
2$-tuples of roots of unity $\alpha(h_i,h_j)$ that satisfy the
condition $|\alpha(h_i,h_j)|$ divides $g.c.d(n_i,n_j)$. In the
sequel we will use freely the correspondence $f\longrightarrow
\alpha_{f}$.

Given a cohomology class $\alpha$ on an abelian group $H$ it is
convenient to attach to it a specific representative $2$-cocycle
$f$. This will permit us to be precise when multiplying elements in
a twisted group algebra. Consider the decomposition above of $H$
into cyclic subgroups and let
$h=h_{1}^{r_1}h_{2}^{r_2}\cdot\cdots\cdot h_{q}^{r_q}$ and
$h^{'}=h_{1}^{s_1}h_{2}^{s_2}\cdot\cdots\cdot h_{q}^{s_q}$ be
arbitrary elements. Set
$f_{\alpha}(h,h^{'})=f_{\alpha}(h_{1}^{r_1}h_{2}^{r_2}\cdot\cdots\cdot
h_{q}^{r_q}, h_{1}^{s_1}h_{2}^{s_2}\cdot\cdots\cdot
h_{q}^{s_q})=\alpha(h_2, h_1)^{r_2s_1}\alpha(h_3,
h_1)^{r_3s_1}\cdot\cdots\cdot \alpha(h_q, h_1)^{r_qs_1}\alpha(h_3,
h_2)^{r_3s_2}\alpha(h_4, h_2)^{r_4s_2}\cdot\cdots\cdot \alpha(h_q,
h_2)^{r_qs_2}\cdot\cdots\cdot \alpha(h_q, h_{(q-1)})^{r_qs_{(q-1)}}$
(each exchange of "variables" $(h_i, h_j)$ yields a factor
$\alpha(h_j, h_i))$. We remark that unlike the function $\alpha$,
the values of the cocycle $f_{\alpha}$ depend on the choice of
generators of $H$ and even on their order.

Let $H$ be as above and let $F^{f}H$ be a twisted group algebra. Let
$\alpha=\alpha_{f}$ and let $f_{\alpha}$ be the canonical
$2$-cocycle that corresponds to $\alpha$ (with respect to the given
ordered set of generators). We assume as we may that $f=f_{\alpha}$.
Consider the subgroup $N$ of $H$ consisting of all elements  $n\in
H$ such that $\alpha(n,h)=1$ for every $h\in H$. By the above
correspondences the restriction of $f$ to $N$ is trivial. In fact
since the value of $f(nh_i,n^{'}h_j)$, $n, n^{'} \in N $ is
independent of $n, n^{'}$, we have that $f$ is inflated from a
$2$-cocycle $\bar{f}:H/N\times H/N \longrightarrow F^{*}$. This
means that $f(h,h^{'})=\bar{f}(\bar{h},\bar{{h}^{'}})$ where $h,
h^{'}$ are arbitrary elements in $H$ and $\bar{h},\bar{{h}^{'}}$ are
the corresponding images in $H/N$.

With the above notation we have

\begin{lemma}
The twisted group algebra $F^{\bar{f}}H/N$ is isomorphic to the
matrix algebra $M_{r}(F)$ where $r^{2}=|H/N|$.
\end{lemma}

\begin{proof}
Since $F^{\bar{f}}H/N$ is semisimple the result will follow if we
show that the center is $1$-dimensional. To see this note that by
the definition of $N$, a representative $u_{\bar{h}}$ in
$F^{\bar{f}}H/N$ for an element $\bar{h}\neq 1$ in $H/N$ cannot be
central. Using the fact that $H/N$ is abelian it follows that
$z=\sum _{\bar{h}\in H/N}\lambda_{\bar{h}}u_{\bar{h}}$ is central if
and only if $z=\lambda_{\bar{1}}u_{\bar{1}}$.
\end{proof}

\begin{definition}
Let $H$ be a finite group and $f$ a $2$-cocycle with coefficients in
$F^{*}$ ($F$ is algebraically closed). We say that $f$ is non
degenerate if the corresponding twisted group algebra is isomorphic
to a a full matrix algebra. A group $H$ admitting a non-degenerate
$2$-cocycle (or non-degenerate cohomology class $[f]$) is said to be
of central type.
\end{definition}

\begin{remark}
Clearly, the abelian group $H/N$ above is of central type and the
$2$-cocycle ${\bar{f}}$ is non-degenerate.
\end{remark}

\begin{corollary}\label{iso}
Let $\{h_{i}\}$ be a transversal of $N$ in $H$. Let $\{u_{h}\}$ be
representatives of $H$ in the twisted group algebra $F^{f}H$ and
$\{v_{\bar{h_{i}}}\}$ be representatives of elements of $H/N$ in
$F^{\bar{f}}H/N$. Let $h=nh_{i}$. Then the map $u_{h}\longmapsto
n\otimes v_{\bar{h_{i}}}$ induces an isomorphism of $F^{f}H$ onto
$FN\otimes F^{\bar{f}}H/N$. In particular $F^{f}H\cong
M_{r}(F)\oplus \cdots\oplus M_{r}(F)$ ($m=|N|$ copies).
\end{corollary}

\begin{proof}
This is now clear.
\end{proof}

Combining this with the structure Theorem \ref{BSZ} we have:

\begin{corollary} (see also \cite{BSZ})
Let $A$ be $G$-simple where $G$ is a finite abelian group. Then $A
\cong M_{k}(M_{r}(F))\oplus \cdots\oplus M_{k}(M_{r}(F)) \cong
M_{kr}(F)\oplus \cdots\oplus M_{kr}(F)$ ($m$-copies).
\end{corollary}
The following proposition is key for the construction of the
required polynomial.

\begin{proposition}\label{main1}
Let $A \cong F^{f}H\otimes M_{k}(F)$ be $G$-simple and let
$\Omega=\{b_h\otimes E_{i,j}\} _{1\leq i,j\leq k,h\in H}$ be a basis
of $A$ (its order is $mr^{2}k^{2}$). Then the set $\Omega$ can be
decomposed into $m$ disjoint sets $\Omega(1),\ldots,\Omega(m)$, each
consisting of $r^{2}k^{2}$ elements with the following properties:
\begin{enumerate}
\item \label{main1_1}
For each $g$ in $G$ the basis elements $b_h\otimes E_{i,j}$ in
$A_{g}$ (i.e. $g_i^{-1}hg_{j}= hg_i^{-1}g_{j}=g)$ are contained in
one and only one set $\Omega(p), p=1,\ldots,m$.

\item
Consider the decomposition $A \cong M_{kr}(F)\oplus\cdots\oplus
(M_{kr}(F)$ ($m$-copies). Then for every $p=1,\ldots,m$ the
projection of $\Omega(p)$ into each component is a basis of $
M_{kr}(F)$.
\end{enumerate}
\end{proposition}

\begin{proof}
Let $\{h_{1}^{p},\ldots,h_{r^2}^{p}\}$, $p=1,\ldots,m$ be $m$ sets
of transversals of $N$ in $H$, pairwise disjoint (i.e. their union
consists of all elements of $H$). Let
$\{\sigma_{1},\ldots,\sigma_{t}\}$ be a transversal of $H$ in $G$.
Multiplying the transversal for $H$ in $G$ with the different
transversals of $N$ in $H$ we obtain $m$ pairwise disjoint
transversals of $N$ in $G$ of the form $\{h_{1}^{p}\sigma_{1},
h_{2}^{p}\sigma_{1},\ldots,h_{r^2}^{p}\sigma_{1},h_{1}^{p}\sigma_{2}
,\ldots,h_{r^2}^{p}\sigma_{2},\ldots,h_{r^2}^{p}\sigma_{t}\}$,
$p=1,\ldots,m$. Denote these transversals (of $N$ in $G$) by
$T_1,\ldots,T_m$. Note that the union of the $T_i$'s is all of $G$.

The main step in the construction is:
\begin{definition}
For every $p=1,\ldots,m$ we let $\Omega(p)$ be the set of all basis
elements $b_h\otimes E_{i,j}$ in $A$ such that
$hg_{i}^{-1}g_{j}=h_{m}^{p}\sigma_{d}$ for some
$m\in\{1,\ldots,r^2\}$ and $d\in\{1,\ldots,t\}$.
\end{definition}

Note that condition 1 of Proposition \ref{main1} follows at once
from the definition. In order to prove condition 2 we need the
following proposition.

\begin{proposition}\label{second_condition}
Consider the set of triples $(h,i,j)$ that appear in basis elements
of $\Omega(p)$. Then the following holds:

\begin{enumerate}
\item
Every pair $(i,j)$ where $1\leq i,j \leq k$ appears in $\Omega(p)$
precisely $r^{2}$ times. In particular the order of $\Omega(p)$ is
$r^2k^2$.
\item
For each $(i,j)$ the set of elements $h$ that appear in some
$b_h\otimes E_{i,j}\in \Omega(p)$ is a complete transversal of $N$
in $H$.
\end{enumerate}
\end{proposition}

\begin{proof}
Assume the pair $(i,j)$ appears in some element in $\Omega(p)$, say
$b_h\otimes E_{i,j}$. This means that
$hg_{i}^{-1}g_{j}=h_{l}^{p}\sigma_{d}$ for some
$l\in\{1,\ldots,r^2\}$ and $d\in\{1,\ldots,t\}$. Then for any $v\in
\{1,\ldots,r^2\}$ we have
$h_{v}^{p}(h_{l}^{p})^{-1}hg_{i}^{-1}g_{j}=h_{v}^{p}\sigma_{d}$.
Now, since the set $\{h_{v}^{p}\mid {1\le v\le r^2}\}$, is a
transversal of $N$ in $H$, the set $\{h_{v}^{p}(h_{m}^{p})^{-1}h\mid
v=1,\ldots,r^2\}$ is a transversal as well. To show that any pair
$(i,j)$ appears in $\Omega(p)$ write $g_{i}^{-1}g_{j}=
nh_{v}^{p}\sigma_{d}$ where $n\in N$ and $h_{v}^{p}\sigma_{d}$ is a
transversal element in $T_p$. Multiplying by $n^{-1}$ the result
follows. In order to complete the proof we need to show that for any
$(i,j)$ if $b_h\otimes E_{i,j}$ and $b_h^{'}\otimes E_{i,j}$ are
different elements that belong to $\Omega(p)$ then $h$ and $h^{'}$
are not congruent modulo $N$. But this is clear for otherwise we
would have in $T_p$ two different elements $hg_{i}^{-1}g_{j}$ and
$h^{'}g_{i}^{-1}g_{j}$ which are congruent modulo $N$ and this
contradicts the definition of $T_p$. This completes the proof of the
proposition.
\end{proof}
Thus the set $\Omega(p)$ consists of the elements ${b_h\otimes
E_{i,j}}$ where $1\leq i,j \leq k$ and for each $(i,j)$ the set of
$h$'s is a transversal of $N$ in $H$ (which depends on the pair
$(i,j)$).

Now, we are ready to show that the sets $\Omega(p)$, $p=1,\ldots,m$,
satisfy condition 2 of Proposition \ref{main1}. By Theorem \ref{BSZ}
and Corollary\ref{iso} we have
$$
A\cong F^{f}H\otimes M_{k}(F) \cong FN\otimes F^{\bar{f}}H/N \otimes
M_{k}(F)
$$
$$
\cong M_{kr}(F)\oplus \cdots\oplus M_{kr}(F) \quad
(m\mbox{-copies}).
$$
Clearly, the projections onto the different components are induced
by multiplying with a primitive idempotent $e_l$ ($l=1,\ldots,m$) of
$FN$ (which is clearly central in $A$). Since the set of $h$'s
appearing in $b_h\otimes E_{i,j}\in \Omega(p)$ for a given $(i,j)$
is a transversal of $N$ in $H$, it follows that $\Omega(p)$ is a
basis of $F^{\bar{f}}H/N \otimes M_{k}(F) \cong M_{kr}(F)$ as a free
module over its center $FN$. If we multiply with a primitive
idempotent $e_l$ of $FN$ it is clear that the set $\Omega(p)e_l$ is
a basis of $FNe_l \otimes F^{\bar{f}}H/N \otimes M_{k}(F) \cong
M_{kr}(F)$. This completes the proof of Proposition\ref{main1}.
\end{proof}

By mean of the previous propositions we shall now show how to
construct, for every $t\ge 1$, a $2t$-fold $(p_1, \ldots,
p_s)$-alternating polynomial, where $(p_1, \ldots, p_s)$ is the
$G$-dimension of $A$,   which is not an identity of $A$ and takes a
central invertible value in $A$. Let
$\varphi(x_{1},\ldots,x_{k^2r^2},y_{1},\ldots,y_{k^2r^2})$ be a
Regev polynomial  for $M_{kr}(F)$. Recall that
$$
\varphi= \displaystyle{\sum_{\sigma,\tau \in S_{k^2r^2}
}({\text{sgn}}\,\sigma \tau) x_{\sigma (1)}y_{\tau (1)} x_{\sigma
(2)}x_{\sigma (3)}x_{\sigma (4)}y_{\tau (2)}y_{\tau (3)} y_{\tau
(4)}}$$ $$ \cdots x_{\sigma (k^2r^2-2kr+2)} \cdots  x_{\sigma
(k^2r^2)} y_{\tau (k^2r^2-2kr+2)} \cdots  y_{\tau (k^2r^2)},
$$
and by \cite{for}, $\varphi$ is a central polynomial but not a
polynomial identity for $M_{kr}(F)$. Moreover $\varphi$ is
alternating in the $x$'s and the $y$'s respectively. In particular
its evaluation on a basis of $M_{kr}(F)$ gives a central non-zero
element.

Now, take $\Omega(p)$ and consider
$\varphi(x_{1},\ldots,x_{k^2r^2},y_{1},\ldots,y_{k^2r^2})$ as a
$G$-graded polynomial by assigning its variables the homogeneous
degrees of the elements of $\Omega(p)$. Rename this polynomial
$\varphi_p$. By applying Proposition \ref{main1} we see that
evaluating the $G$-graded polynomial with elements of $\Omega(p)$ we
obtain a central and invertible value in $M_{kr}(F)\oplus
\cdots\oplus M_{kr}(F)$. Consider now the product of $m$ $G$-graded
Regev polynomials, one for each set $\Omega(p)$. It follows that:
\begin{enumerate}
\item
the product $\varphi_1\cdots \varphi_m$ is a $G$-graded polynomial
in $2mk^2r^2$ variables;
\item
All $x$'s and all $y$'s with the same homogeneous degree belong to
the same factor and hence they are alternating;
\item
The evaluation of $\varphi_1\cdots \varphi_m$ on a basis of $A$ is a
central and invertible element of $A$;
\item
Multiplying $t$-copies of $\varphi_1\cdots \varphi_m$ we obtain a
polynomial of degree $2t(\dim A)$ which is $2t$-fold $(p_1, \ldots,
p_s)$-alternating.
\end{enumerate}

We have proved the following

\begin{lemma} \label{central}
Let $G$ be a finite abelian group and $A=\oplus_{g\in G}A_{g}$ a
finite dimensional $G$-graded simple algebra over an algebraically
closed field $F$. Let $G$-$\dim A= (p_1, \ldots, p_s)$. Then, for
all $t\ge 1$ there exists a $2t$-fold $(p_1, \ldots,
p_s)$-alternating polynomial
$$\varphi(X^1_{g_1}, \ldots, X^1_{g_s},
\ldots, X^t_{g_1}, \ldots, X^t_{g_s})
$$ which is not an identity
for $A$ and takes an invertible central value in $A$.
\end{lemma}

\section{Computing the exponential growth of the $G$-codimensions}

We recall our general setting: $A=\oplus_{g\in G}A_{g}$ is a finite
dimensional $G$-graded algebra over an algebraically closed field
$F$ and $G=\{g_1=1, \ldots, g_s\}$ is an abelian group. Also $A=
B_1\oplus\cdots\oplus B_m+J$ where the $B_i$'s are $G$-simple graded
algebras and $d=d(A)$ is the integer defined in (\ref{d}).

Given a polynomial $f=f(X_1, y_1, \ldots, y_r)$ which is multilinear
on the set $X_1=\{x_1, \ldots, x_t\}$ we define the operator of
alternation  $\mathcal{A}_{X_1}$ as
$$
\mathcal{A}_{X_1}(f)=\sum_{\sigma\in S_t}(\mbox{sgn}\,
\sigma)f({x_{\sigma(1)}, \ldots, x_{\sigma(t)}, y_1, \ldots, y_r}).
$$
Notice that the new polynomial $\mathcal{A}_{X_1}(f)$ is alternating
on $X_1$.

\begin{lemma}
Let $C=C_1\oplus\cdots\oplus C_k$ be a semisimple $G$-graded
subalgebra of $A$, where $C_1, \ldots, C_k\in \{B_1, \ldots, B_m\}$
are distinct, and $C_1JC_2J\cdots JC_k \ne 0$. Let also $G$-$\dim
C=(q_1, \ldots, q_s)$. Then for all $t\ge 1$, there exists a
$2t$-fold $(q_1, \ldots, q_s)$-alternating polynomial
$$
\varphi(X^1_{g_1}, \ldots, X^1_{g_s}, \ldots, X^{2t}_{g_1}, \ldots,
X^{2t}_{g_s},Y),
$$
where $|Y|=2k-1$, not vanishing on $A.$
\end{lemma}

\begin{proof} For $i=1, \ldots, k,$ let  $G$-$\dim C_{i}=(q^i_{1}, \ldots, q^i_{s})$.
Let also $\varphi_i$ be the $2t$-fold $(q^i_{1}, \ldots,
q^i_{s})$-alternating polynomial in the sets $\{X^{1}_{i,
g_1},\ldots, X^{1}_{i, g_s}, \ldots, X^{2t}_{i, g_1},\ldots,
X^{2t}_{i, g_s}\}$ constructed in Lemma \ref{central}.

Define
$$
\varphi=y_1\varphi_1z_1y_2\varphi_2z_2\cdots \varphi_{k-1}z_{k-1}y_k
\varphi_k
$$
and, for $l=1, \ldots, 2t,$ set $X^{l}_{g_j}=\bigcup_{i=1}^k
X^{l}_{i, g_j}$.  Clearly $|X^{l}_{g_j}|=\sum_{i=1}^k q^i_{j}=q_j.$

Let
$$
\varphi'=\mathcal{A}_{X^{1}_{g_1}}\cdots
\mathcal{A}_{X^{2t}_{g_1}}\cdots \mathcal{A}_{X^{1}_{g_s}}\cdots
\mathcal{A}_{X^{2t}_{g_s}}\varphi.
$$
Since $C_1JC_2J\cdots C_{k-1}JC_k \ne 0$ there exist $c_i\in C_i,
1\leq i \leq k,$ $b_1, \ldots, b_{k-1}\in J$ such that
$c_1b_1c_2b_2\cdots b_{k-1}c_k\neq 0.$ Since by Lemma \ref{central}
each $\varphi_i$ takes a central invertible value in $C_i$, we may
assume that such value is $\bar\varphi=1_{C_i}.$ Hence, by also
setting $\bar y_i=c_i,$ $\bar z_i=b_i$ we get the following non-zero
evaluation of $\varphi'$
$$
\bar\varphi'=(q^1_{1}!q^2_{1}!\cdots q^k_{1}! \cdots,
q^1_{s}!q^2_{s}!\cdots q^k_{s}!)^{2t}c_1b_1c_2b_2\cdots
b_{k-1}c_k\neq 0.
$$
If $k_i$ is the number of elements  in  $\{c_1, b_1, c_2, b_2,
\ldots, b_{k-1}, c_k\}$ of homogeneous degree $g_i$ then
$$
\varphi'=\varphi'(X^{1}_{g_1}, \ldots, X^{1}_{g_s},
\ldots,X^{2t}_{g_1}, \ldots, X^{2t}_{g_s}, y_{1,g_1},\ldots,
y_{k_1,g_1}, \ldots, y_{1,g_s}, \ldots, y_{k_s,g_s}),
$$
with $k_1+\ldots+k_s=2k-1,$ is the required $G$-graded polynomial
not vanishing on $A$.
\end{proof}

We can now compute the lower bound of the $G$-codimensions.

\begin{lemma} \label{lem19}
For the finite dimensional $G$-graded algebra $A$ we have
$c_n^{G}(A)\geq a n^{b}d^n,$ for some constants $a>0$ and $b.$
\end{lemma}

\begin{proof}
Let $C_1,\ldots,C_k$  be distinct $G$-graded simple subalgebras of
$B$ such that $$ C_1JC_2\break J\cdots C_{k-1}JC_k \ne 0
$$
and $\dim (C_1\oplus \cdots \oplus C_k)=d=d(A).$ If
$p_i=\dim((C_1)_{g_i}+\cdots +\dim (C_k)_{g_i}), i=1, \ldots, s$,
then $d=p_1+\cdots +p_s$.

Also, for every even $t\ge 2$, let $\varphi_{t}(X^1_{g_1}, \ldots,
X^1_{g_s}, \ldots, X^{t}_{g_1}, \ldots, X^t_{g_s},Y)$ be the
polynomial constructed in the previous lemma where
$$
Y=\{y_{1,g_1},\ldots, y_{k_1,g_1}, \ldots, y_{1,g_s}, \ldots,
y_{k_s,g_s}\}.
$$

Now take any $n\ge 2d+(k_1+\cdots +k_s)$, and divide
$n-(k_1+\cdots+k_s)$ by $2d$; then we write
$$
n=2m(p_1+\cdots +p_s)+(k_1+\cdots+k_s)+t,
$$
for some $m,t$, where $0\le t < 2d$.

Set $n_1= 2mp_1+k_1+t$ and $n_i = 2mp_i+k_i,$ $2\le i\le s$. If
$\varphi=\varphi_{2m}$ is the above polynomial, then $\deg
\varphi=2m(p_1+\cdots+ p_s)+k_1+\cdots+k_s$, and set
$$
\varphi'= \varphi y_{k_1+1,g_1}\cdots y_{k_1+t,g_1}.
$$
Clearly  $\varphi' \in P_{n_1, \ldots, n_s}$, where $n_1+\cdots
+n_s=n,$ and $\varphi'$ does not vanish on $A$.

We let the group $H=S_{2mp_1}\times\cdots\times S_{2mp_s}$ act on
$P_{n_1, \ldots, n_s}$ by letting $S_{2mp_i}$ act on $X^1_{g_i}\cup
\cdots \cup X^{2m}_{g_i}$, $1\le i\le m$.   Let $M=FH\varphi'$ be
the $H$-submodule of $P_{n_1, \ldots, n_s}$ generated by $\varphi'$.
By complete reducibility $M$ splits into the sum of irreducible
$H$-modules. Hence, since $M\not\subseteq \I^G(A)$,  there exists an
irreducible submodule $M'\not\subseteq \I^G(A)$ with
$M'=W_{\langle\lambda\rangle}=FH e_{T_{\lambda(1)}}\cdots
e_{T_{\lambda(s)}}\varphi'$ where for $i=1,\ldots, s$,
$T_{\lambda(i)}$ is a Young tableau of shape $\lambda(i) \vdash
2mp_i$ and $\langle \lambda\rangle=(\lambda(1), \ldots,
\lambda(s))$.

Recall that for every $i$,  $\lambda(i)_1$ is the length of the
first row of $\lambda(i)$. Now, for every $\tau \in S_{2mp_i}$,
$\tau(\varphi')$ is still alternating on $2m$ disjoint sets of
variables of degree $g_i,$ each of these ones containing $p_i$
elements and $\sum_{\sigma \in R_{T_{\lambda(i)}}}\sigma$ acts by
symmetrizing $\lambda(i)_1$ variables. Thus, if $\lambda(i)_1> 2m,$
there would exist at least two variables on which
$$
(\sum_{\sigma \in R_{T_{\lambda(i)}}}\sigma) (\sum_{\tau \in
C_{T_{\lambda(i)}}}(\mbox{sgn}\, \tau)\tau)\varphi' =
e_{T_{\lambda(i)}}\varphi'
$$
is alternating and symmetric. This would say that
$e_{T_{\lambda(i)}}\varphi'=0$, a contradiction. Hence
$\lambda(i)_1\leq 2m$ for all $i=1, \ldots, s$.

Since $\lambda(i)\vdash 2mp_i$, this says that $\lambda(i)_1'\geq
p_i$, for all $i=1, \ldots, s$, where $\lambda(i)_1'$ is the length
of the first column of $\lambda(i)$.

Next we claim that if $r_i$ is the number of boxes out of the first
$p_i$ rows of $\lambda(i)$, then $r_1+\cdots + r_s\le l$ where
$J^{l+1}=0$.

In fact, suppose to the contrary that such number of boxes is at
least $l+1$. Since $FH e_{T_{\lambda(1)}}\cdots e_{T_{\lambda(s)}}$
is a minimal left ideal of $FH$, then
$$FH\bar{C}_{T_{\lambda(1)}}\cdots
\bar{C}_{T_{\lambda(s)}}e_{T_{\lambda(1)}}\cdots e_{T_{\lambda(s)}}=
FH e_{T_{\lambda(1)}}\cdots e_{T_{\lambda(s)}}$$ where
$\bar{C}_{T_{\lambda(i)}}=\sum_{\sigma \in
C_{T_{\lambda(i)}}}(\mathrm{sgn \sigma})\sigma, i=1, \ldots, s.$  We
need to evaluate the polynomial
$\tilde{\varphi'}=\bar{C}_{T_{\lambda(1)}}\cdots
\bar{C}_{T_{\lambda(s)}}e_{T_{\lambda(1)}}\cdots
e_{T_{\lambda(s)}}\varphi'$ on $A$.

For every $i=1, \ldots, s$, write the conjugate partition
$\lambda(i)'$ of $\lambda(i)$ as
$$
\lambda(i)'=(p_i+r_{i,1}, \ldots, p_i+r_{i,u_i},
\lambda(i)_{u_i+1}', \ldots , \lambda(i)_{m_i}').
$$
Here $\sum_{j=1}^{u_i} r_{i,j}=r_i$ and $\lambda(i)_{u_i+1}', \ldots
, \lambda(i)_{m_i}'\leq p_i.$ Now, the polynomial $\tilde{\varphi'}$
is multialternating on $u_i$ sets of variables of order
$p_i+r_{i,1}, \ldots, p_i+r_{i,u_i}$, respectively (corresponding to
the first $u_i$ columns of $\lambda(i)$). If we substitute on any of
these sets of variables only elements from the semisimple subalgebra
$C=C_1 \oplus \cdots \oplus C_k$, we would get zero since $\dim
C_{g_i}=p_i$. It follows that we have to substitute into these sets
of variables at least $$
r_{1,1}+\cdots + r_{1,u_1} +\cdots r_{s,1}
+ \ldots +r_{s,u_s} =r_1+\cdots +r_s\geq l+1
$$
elements form the Jacobson radical $J$. Since $J^{l+1}=0$, we would
get that $\tilde{\varphi'}$ vanishes on $A$. Hence $FH
e_{T_{\lambda(1)}}\cdots e_{T_{\lambda(s)}}\varphi'\subseteq
Id^G(A)$ and this is a contradiction.

We have proved that  $\lambda(1), \ldots, \lambda(s)$ must contain
at most a total number of $l$ boxes out of the first $p_1, \ldots,
p_s$ rows, respectively. Since for every $i=1,\ldots, s$,
$\lambda(i)_1\leq 2m$ the outcome is that each $\lambda(i)$ must
contain the rectangle
$$
\nu_i ={\displaystyle(\underbrace{2m-l, \ldots,
2m-l}_{p_i})=((2m-l)^{p_i})}.
$$
Now, by \cite{reg2}, as $m \rightarrow \infty$,
$$
\deg \chi_{\nu_i}\simeq a_ip_i^{(2m-l)p_i}((2m-l)p_i)^{b_i},
$$
for some constants $a_i,b_i$. Hence, since $l$ is a constant not
depending on $m$, we obtain that
$$\deg
\chi_{\nu_1}\cdots \deg \chi_{\nu_s}\simeq a((2m-l)p_1)^{b_1}\cdots
((2m-l)p_s)^{b_s}p_1^{(2m-l)p_1}\cdots p_s^{(2m-l)p_s}\geq$$$$n^u
p_1^{2mp_1}\cdots p_s^{2mp_s},$$
 for some constants $u, a$, $b_1, \ldots, b_s.$
Now recall that for all $i$, $\lambda(i)\supseteq \nu_i$. Hence,
since
$$
\dim W_{\langle\lambda\rangle}=\deg \chi_{\lambda(1)}\cdots\deg
\chi_{\lambda(s)}\geq \deg \chi_{\nu_1}\cdots \deg \chi_{\nu_s},
$$
we obtain
$$
c_{n_1, \ldots, n_s}(A)\geq \dim W_{\langle\lambda\rangle}\geq
n^u p_1^{2mp_1}\cdots p_s^{2mp_s}.
$$
Recalling the relation between the $G$-codimensions $c^G_n(A)$ and
$c_{n_1,\ldots, n_s}(A)$ given in (\ref{cG}), we get
$$
c^G_n(A)=\sum_{n_1+\cdots+n_s=n}\binom{n}{n_1, \ldots,
n_s}c_{n_1,\ldots, n_s}(A)\geq \binom{n}{n_1, \ldots, n_s}c_{n_1,
\ldots, n_s}
$$
$$
\geq n^u p_1^{2mp_1}\cdots p_s^{2mp_s}\frac{n!}{n_1!\cdots n_s!}.
$$
Clearly since $n_1=2mp_1+k_1+t$ and $n_i=2mp_i+k_i,$ $i=2, \ldots,
s$, we have that
$$
\frac{n!}{n_1!\cdots n_s!}\geq \frac{(2mp_1+\cdots+
2mp_s)!}{2mp_1!\cdots 2mp_s!}.
$$
We now invoke Stirling formula  asserting that (see \cite{robbins})
$$
n! = \sqrt{2\pi n} \left( \frac{n}{e}\right)^n
e^{\frac{\theta_n}{12n}},
$$
for some $0\le \theta_n\le 1$. We get
$$
c^G_n(A)\geq n^\alpha \frac{(2mp_1+\cdots +2mp_s)^{2mp_1+\cdots+
2mp_s}}{(2mp_1)^{2mp_1}\cdots(2mp_s)^{2mp_s}}p_1^{2mp_1}\cdots
p_s^{2mp_s}
$$
$$
=n^\alpha(p_1+\cdots +p_s)^{2mp_1+\cdots+2mp_s} =an^\alpha
(p_1+\cdots +p_s)^n=an^\alpha d^n,
$$
where $a=(p_1+\cdots+p_s)^{-(k_1+\cdots+k_s+t)}$ and $\alpha$ is a
constant. This completes the proof of the lemma.
\end{proof}

Putting together Lemma \ref{lem6} and Lemma \ref{lem19} we get

\begin{theorem} \label{mainth} Let $F$ be an algebraically closed field
of characteristic zero and $G$ a finite abelian group. If $A$ is a
finite dimensional $G$-graded algebra, then there exist constants
$C_1>0, C_2, k_1,k_2$ such that
$$
C_1n^{k_1}d^n \le c^G_n(A) \le C_2n^{k_2}d^n,
$$
where $d$ equals the dimension of some semisimple $G$-graded
subalgebra of $A$. In particular $\exp^G(A) =
\displaystyle{\lim_{n\rightarrow \infty}}\root n \of{c_n^{G}(A)}=d$
exists and is a non-negative integer.
\end{theorem}

Recalling that by extending the ground field, the graded
codimensions do not change we get the following

\begin{corollary} \label{main}
Let $F$ be a  field of characteristic zero and $G$ a finite abelian
group. For any finite dimensional $G$-graded algebra $A$, $\exp^G(A)
= \displaystyle{\lim_{n\rightarrow \infty}}\root n \of{c_n^{G}(A)}$
exists and is a non-negative integer.
\end{corollary}

\end{document}